\documentclass[]{scrartcl}

\usepackage[utf8]{luainputenc}
\usepackage[USenglish]{babel}
\usepackage{csquotes}

\usepackage[a4paper,top=27mm,bottom=20mm,inner=25mm,outer=20mm]{geometry}

\usepackage[%
  backend=bibtex,bibencoding=ascii,
  style=numeric-comp,
  giveninits=true, uniquename=init, %abbreviate first names
  natbib=true,
  url=true,
  doi=true,
  isbn=false,
  backref=false,
  maxnames=99,
  ]{biblatex}
\usepackage{amsmath}
\allowdisplaybreaks
\numberwithin{equation}{section}
\usepackage{amssymb}
\usepackage{commath}
\usepackage{mathtools}
\usepackage{bbm}
\usepackage{nicefrac}
\usepackage{subdepth}

\usepackage{siunitx}
\usepackage{algorithm}
\usepackage{algpseudocode}

\usepackage{amsthm}
\usepackage{thmtools}
\usepackage{etoolbox}
\makeatletter
\patchcmd{\thmt@setheadstyle}
 {\bgroup\thmt@space}
 {\thmt@space}
 {}{}
\patchcmd{\thmt@setheadstyle}
 {\egroup\fi}
 {\fi}
 {}{}
\makeatother
\declaretheoremstyle[
  bodyfont=\normalfont\itshape,
  headformat=\NAME\ \NUMBER\NOTE,
]{myplain}
\declaretheoremstyle[
  headformat=\NAME\ \NUMBER\NOTE,
]{mydefinition}
\newcommand{\envqed}{{\lower-0.3ex\hbox{$\triangleleft$}}}
\declaretheorem[style=myplain,numberwithin=section]{theorem}

\declaretheorem[style=mydefinition,numberlike=theorem,qed=\envqed]{definition}
\declaretheorem[style=mydefinition,numberlike=theorem,qed=\envqed]{remark}

\usepackage[plainpages=false,pdfpagelabels,hidelinks,unicode]{hyperref}

\usepackage{color}
\usepackage{graphicx}
\usepackage[small]{caption}
\usepackage{subcaption}

\begingroup\expandafter\expandafter\expandafter\endgroup
\expandafter\ifx\csname pdfsuppresswarningpagegroup\endcsname\relax
\else
\fi

\usepackage{booktabs}
\usepackage{rotating}
\usepackage{multirow}

\usepackage{enumitem}

\usepackage{ifluatex}
\ifluatex
  \usepackage[no-math]{fontspec}
\else
  \usepackage[T1]{fontenc}
\fi
\usepackage{newpxtext,newpxmath}

\newcommand{\e}{\mathrm{e}}
\newcommand{\R}{\mathbb{R}}

\newcommand{\orcid}[1]{ORCID:~\href{https://orcid.org/#1}{#1}}
\usepackage{authblk}

\newenvironment{keywords}{\par\textbf{Key words.}}{\par}
\newenvironment{AMS}{\par\textbf{AMS subject classification.}}{\par}

\title{Modeling still matters: a surprising instance of catastrophic floating point errors in mathematical biology and numerical methods for ODEs} %TODO
\date{October 05, 2024} 

\author[1]{Cordula Reisch\thanks{\orcid{0000-0003-1442-1474}}}
\affil[1]{Institute for Partial Differential Equations, TU Braunschweig, Germany}

\author[2]{Hendrik Ranocha\thanks{\orcid{0000-0002-3456-2277}}}
\affil[2]{Applied Mathematics, University of Hamburg, Germany
  Present address: Institute of Mathematics, Johannes Gutenberg University Mainz, Germany. }

\begin{document}

\maketitle

\begin{abstract}
\noindent
  We guide the reader on a journey through mathematical modeling and 
numerical analysis, emphasizing the crucial interplay of both disciplines.
Targeting undergraduate students with basic knowledge in dynamical
systems and numerical methods for ordinary differential equations,
we explore a model from mathematical biology where numerical methods
fail badly due to catastrophic floating point errors. We analyze the
reasons for this behavior by studying the steady states of the model
and use the theory of invariants to develop an alternative model that
is suited for numerical simulations. Our story intends 
to motivate combining analytical and numerical know\-ledge, even in 
cases where the world looks fine at first sight. We have set up an
online repository containing an interactive notebook with all
numerical experiments to make this study fully reproducible and
useful for classroom teaching.

\end{abstract}

%TODO: keywords
\begin{keywords}
mathematical modeling,
  steady states,
  dynamical systems,
  numerical instability,
  invariants,
  Runge-Kutta methods
\end{keywords}

%TODO: MSC
\begin{AMS}
  37M05,  % Dynamical systems and ergodic theory, Approximation methods and numerical treatment of dynamical systems: Simulation of dynamical systems
  65L06,  % NA, ODEs: Multistep, Runge-Kutta and extrapolation methods
  65L20,  % NA, ODEs: Stability and convergence of numerical methods for ordinary differential equations
  65P40,  % NA, Numerical problems in dynamical systems: Numerical nonlinear stabilities in dynamical systems
  97D40   % Mathematics education, Education and instruction in mathematics: Mathematics teaching methods and classroom techniques
\end{AMS}

\section{Introduction}

Setting up a nice mathematical model for some applications, using a proper numerical scheme for solving the differential equations --- and wondering what is happening?
Oftentimes, mathematical models of ordinary differential equations (ODEs) are easily solvable numerically as long as a suitable scheme is used.
In this report, we tell a different story, which highlights the process of setting up a model and analyzing it properly for ensuring trustable and correct numerical solutions.

This article intends to take the reader on a small journey through the interplay between modeling, numerical simulation and analysis.
We will pass a biological motivation on the inheritance of genes like the inheritance of flower colors, and set up a mathematical model of ODEs that describes the proportions of different genotypes in the total population.
As a next step, we try to get a first impression of the system's behavior by running numerical simulations that are expected to show the time-dependent dynamics of the system.
The first plot twist is on the goodwill and trust concerning numerical simulation, even when the schemes are chosen wisely.
We discover the reasons for the strange simulations by changing the focus of research from numerical analysis to dynamical system analysis and analyzing the steady states of the dynamical system.
The curtains fall for the theory of invariants.
The friendly-looking model for inheritance turns out to conserve an important quantity --- the total population --- as a second integral.
Accumulating floating point errors reveal the instability of steady states.
By changing the system, we turn the total population into a first integral. The new model reacts in a stable way to floating point errors and trust in numerical simulations is restored.

Please use our story to motivate combining analytical and numerical knowledge, even in cases where no problems are expected. 
A related story focusing on step size control of
numerical ODE solvers is published by
Skufca \cite{skufca2004analysis}, which we used
as an inspiration for the title of this report.
To make this study fully reproducible, we have set up an
online repository \cite{reisch2023modelingRepro}
containing an interactive notebook with all numerical
experiments using the modern programming language Julia.

Before starting our journey, we would like to highlight the aim of this article. 
The intention is not to provide new mathematical results but to motivate students from a modeling perspective why a deeper dive into mathematical analysis is helpful and provides much insight.
Therefore, we refer to analytical results rather from a modeling perspective emphasizing the nature of the problems, and we include references for further studies in the end.
The outlook will include literature on the theory of ordinary differential equations on manifolds with local coordinates and numerical correction schemes, further modeling examples, and structure-preserving numerical methods.
During the story, the focus will be on the applied modeling perspective.

\section{Modeling the inheritance of genes}

Mathematical biologists strive for explaining life in mathematical terms by abstracting from individuals and finding general insight into complex processes. 
One big question of life is the inheritance of genes and therefore the evolution of populations. 
Here, we model how genes are passed on from a parent generation to the next generation. 
While reading, you can have in mind flowers with different possible colors like white, rose and red, 
or any example where two copies of a gene are present. 
Such an organism is called diploid. 
Different versions of a gene are called alleles.
We consider here two different alleles, either $X$ or $x$.
Possible combinations give the genotypes $XX$, $Xx$ and $xx$.
If every parent inherits one of its alleles, the genotypes of the next generation are given according to the probabilities of the alleles in the parent generation.
The formulation of generations implies a discrete time, see for example \cite[Chapter 4]{britton_essential_2003}. In this setting, we assume all processes to be time continuous with blurred generations.  
The assumption of blurred generations in continuous time is in particular relevant for longer timescales. On short timescales, these assumptions are too strong and lead to an unfeasible model. Here, we are interested in describing the genetic drift, which is the change of the allele and gene distribution in a population.

In this model, the population is divided into compartments $c_i$ according to their genotypes.
Let $c_1$ be the compartment of the genotype $XX$,
$c_2$ the compartment for the mixed genotype $Xx$,
and $c_3$ the compartment of $xx$.
The size of each compartment is time-depending.

The inheritance of the genotypes is given by Mendel's Law of Segregation in Table~\ref{tab:Mendel}.
The genotypes of the outcomes depend on the genotypes of the parents' generation: 
Every offspring inherits one allele of each parent, e.g., an offspring has the genotype $c_1$ if parents with $XX$ alleles or $Xx$ alleles mix. 
The probabilities of passing on the genotypes are given in Table~\ref{tab:Mendel}.

\begin{table}[h] 
\caption{Mendel's Law of Segregation: Probability of the genotypes of the offsprings depending on the parents genotypes $XX$, $Xx$, $xx$ with their compartments $c_1, c_2, c_3$.}
\label{tab:Mendel}
\begin{tabular}{l|lll}
            & $c_1\ [XX]$          & $c_2\ [Xx]$                     & $c_3\ [xx]$          \\ \hline
$c_1\ [XX]$ & $c_1$: 1              & $c_1$: 1/2, $c_2$: 1/2           &  $c_2$:1              \\[0.5em] %\hline
$c_2\ [Xx]$ & $c_1$:1/2,  $c_2$:1/2 \quad \, &  $c_1$: 1/4,  $c_2$: 1/2,  $c_3$:1/4 \quad \, &  $c_2$: 1/2, $c_3$:1/2 \\[0.5em] %\hline
$c_3\ [xx]$ &  $c_2$:1              &  $c_2$: 1/2,  $c_3$: 1/2            &  $c_3$: 1          
\end{tabular}
\end{table}
Table \ref{tab:Mendel} can be read in the example of flower colors as follows: If $XX$ is the red flower genotype, $c_1$ gives the number of red flowers; $Xx$ is the rose genotype with compartment $c_2$ and $xx$ is the white genotype with compartment $c_1$. If the parents are of genotype $XX$ and $Xx$, the offspring is red with probability 1/2 (genotype $c_1$, $XX$), and rose with probability 1/2 (genotype $c_2$, $Xx$). 

Following these mixing rules, Mendel's Law of Segregation can be transformed into ordinary differential equations, compare \cite{langemann2012multi}, for the genotype compartments $c_i$, namely
\begin{equation}
\label{eq:compartments}
\begin{aligned}
    c_1' &= \frac{g(c_\mathrm{sum})}{c_\mathrm{sum}^2} \left (c_1^2 + c_1 c_2 + \frac{1}{4} c_2^2 \right), \\
    c_2' &= \frac{g(c_\mathrm{sum})}{c_\mathrm{sum}^2} \left (\frac{1}{2} c_2^2 + c_1 c_2 + 2 c_1 c_3 + c_2 c_3\right), \\
    c_3' &= \frac{g(c_\mathrm{sum})}{c_\mathrm{sum}^2} \left (\frac{1}{4} c_2^2 + c_2 c_3 + c_3^2\right),
\end{aligned}
\end{equation}
where $g(c_\mathrm{sum})\geq 0$ is a growth function for the whole population $c_\mathrm{sum}=c_1+c_2+c_3$.

In model \eqref{eq:compartments}, all genotypes have the same growth rate.
The factor $1/c_\mathrm{sum}^2$ normalizes the equations such that the overall growth is a function of the total population $c_\mathrm{sum}$,
\begin{equation}
    c_1' +c_2' +c_3' = \frac{g(c_\mathrm{sum})}{c_\mathrm{sum}^2} \left ( c_1^2  + c_2^2 + c_3^2 + 2c_1c_2 + 2 c_1 c_3 + 2 c_2c_3 \right ) = \frac{g(c_\mathrm{sum})}{c_\mathrm{sum}^2} (c_1 +c_2 + c_3)^2,
\end{equation}
resulting in $c_\mathrm{sum}' = g(c_\mathrm{sum})$.
Other dynamics of the compartments like mutation, mortality, or fitness are not regarded here, but for example in Langemann, Richter and Vollrath \cite{langemann2012multi}, Britton \cite{britton_essential_2003}, and Allen and McAvoy\cite{allen_mathematical_2019}.
Those variations of models include for example the assumption of genotype-specific growth rates depending on the fitness which effects the assumptions in \eqref{eq:compartments}.

The equations in system \eqref{eq:compartments} can be written as well in a matrix vector formalism \cite{langemann2012multi} using $c=(c_1, c_2, c_3)^T$ and the inheritance matrices
\begin{equation}
\label{eq:matrices}
    \renewcommand*{\arraystretch}{1.2}
    W_1= \begin{pmatrix}
        1 & \frac{1}{2} & 0 \\
        \frac{1}{2} & \frac{1}{4} & 0 \\
        0 & 0 & 0
    \end{pmatrix}, \quad
     W_2= \begin{pmatrix}
        0 & \frac{1}{2} & 1 \\
        \frac{1}{2} & \frac{1}{2} & \frac{1}{2} \\
        1 & \frac{1}{2} & 0
    \end{pmatrix}, \quad
        W_3= \begin{pmatrix}
          0 & 0 & 0\\
        0 & \frac{1}{4} & \frac{1}{2} \\
        0 & \frac{1}{2} & 1
    \end{pmatrix}
\end{equation}
with $W_1 +W_2 + W_3= \mathbbm{1}_{3 \times 3}$
(the matrix in $\mathbb{R}^{3 \times 3}$ with all
entries equal to unity).
Then
\begin{equation}
\label{eq:compartmentvector}
    c_i'= \frac{g(c_\mathrm{sum})}{c_\mathrm{sum}^2} c^TW_i c
\end{equation}
gives system \eqref{eq:compartments}.

Due to the growth function $g(c_\mathrm{sum})>0$, the total population grows. 
Regarding long-term phenomena, this assumption of unlimited growth is unrealistic. Therefore, and to focus on the drift effect of the mutations, we change the modeling objective to a situation with conserved quantities.
From a biological point of view, both the absolute number of copies and proportions of genotypes are of interest.
The genotype proportions $q_i= c_i/c_\mathrm{sum}$ at a certain time $t$ can be estimated by collecting and evaluating samples.
Using the product rule and \eqref{eq:compartmentvector}, the dynamics of the proportions $q_i$ are given indirectly by
\begin{equation}
    c_i'= q_i' c_\mathrm{sum} + q_i c_\mathrm{sum}'.
\end{equation}
The change of the total population is given by the growth function, $c_\mathrm{sum}'=g(c_\mathrm{sum})$ and the change of $c_i$ is given by \eqref{eq:compartmentvector}.
Therefore, the genotype proportions $q=(q_1, q_2,q_3)^T$ follow
\begin{equation}
%\begin{aligned}
    q_i'= \frac{g(c_\mathrm{sum})}{c_\mathrm{sum}^3} \left ( c^TW_i c   \right) - q_i \frac{g(c_\mathrm{sum})}{c_\mathrm{sum}}
        =\frac{g(c_\mathrm{sum})}{c_\mathrm{sum}} \left ( q^TW_i q  - q_i \right).
%\end{aligned}
\end{equation}
For the special choice of exponential growth $g(c_\mathrm{sum})=c_\mathrm{sum}$, this becomes
\begin{equation}
\label{eq:quantityvector}
    q_i'= f_i(q) =  q^TW_i q -q_i, \quad i \in \{1,2,3\},
\end{equation}
 or written as a system
 \begin{equation}
\label{eq:system3-original}
\begin{aligned}
    q_1' &= q_1^2 + q_1 q_2 + \frac{1}{4} q_2^2 - q_1, \\
    q_2' &= \frac{1}{2} q_2^2 + q_1 q_2 + 2 q_1 q_3 + q_2 q_3 - q_2, \\
    q_3' &= \frac{1}{4} q_2^2 + q_2 q_3 + q_3^2 - q_3.
\end{aligned}
\end{equation}
The initial conditions $q_1(0)=q_{01}$, $q_2(0)=q_{02}$ and $q_3(0)=q_{03}$ should fulfill $q_i(0) \in [0,1]$ and 
$$q_\mathrm{sum}(0)=\sum_{i=1}^3 q_i(0)=1$$ since
$q_i$ describe the proportions of the different genotypes in the whole population.
We investigate the time-dependent dynamics of the genotypes in the population.

There are basically two paths one can follow from here on.
The ``ideal'' way begins with a mathematical analysis of the model;
if everything is well-understood qualitatively and the model behaves
as expected, numerical simulations can be used with confidence to
gain a quantitative understanding of the dynamics.
Here, we follow the ``practical'' approach chosen probably in most
cases by practitioners: We directly perform numerical simulations
to get a basic understanding of the model; based on the numerical
results, we will focus our attention on interesting behavior for
a deeper mathematical analysis.

\subsection{Numerical experiments}

All numerical experiments presented in this article are available
from our reproducibility repository \cite{reisch2023modelingRepro}.
In particular, an interactive Pluto.jl \cite{pluto} notebook
based on Julia \cite{bezanson2017julia} is available for download.

Numerical simulations provide a first impression of the spread of the genotypes over time.
At first, we use the fifth-order Runge-Kutta method of Tsitouras \cite{tsitouras2011runge},
which is the recommended default method for non-stiff problems in OrdinaryDiffEq.jl
\cite{rackauckas2017differentialequations}.
As shown in Figure~\ref{fig:system3_original_Tsit5}, the numerical solution appears to
converge to a steady state --- until it suddenly goes to zero. 
While the tendency towards the steady state matches our expectation, the decay to zero contradicts
the modeling assumption $\sum_i q_i = 1$ and does not change significantly if we use stricter
tolerances --- the time of ``extinction'' may only be postponed a bit.
\begin{figure}[ht]
\centering
    \includegraphics[width=0.5\textwidth]{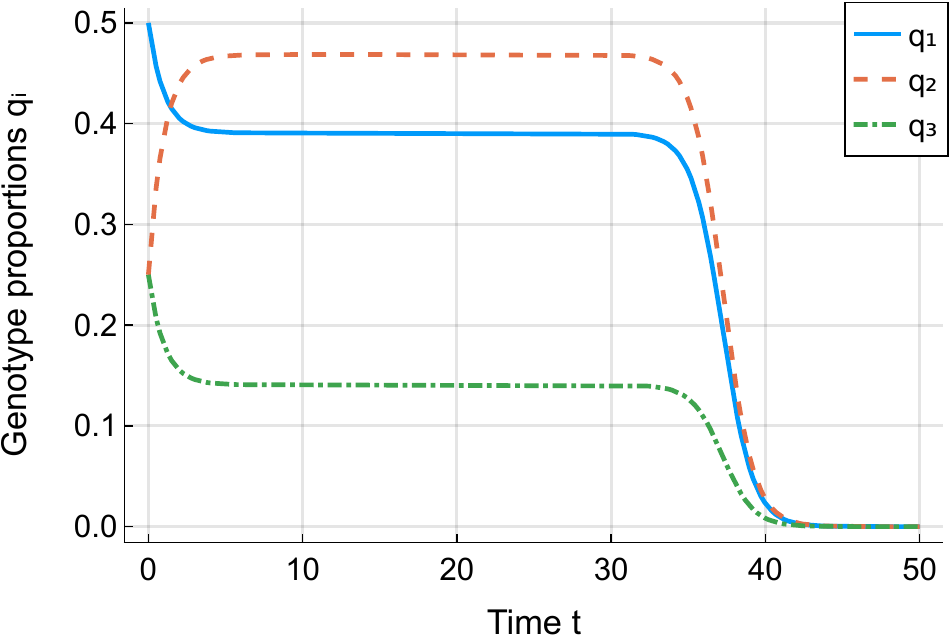}
    \caption{Numerical solution of the system \eqref{eq:system3-original}
             with initial condition $q_0 = (0.5, 0.25, 0.25)^T$ obtained by the
             fifth-order Runge-Kutta method of Tsitouras \cite{tsitouras2011runge}
             implemented in OrdinaryDiffEq.jl \cite{rackauckas2017differentialequations}
             in Julia \cite{bezanson2017julia} with absolute and relative tolerances
             $10^{-8}$.}
    \label{fig:system3_original_Tsit5}
\end{figure}

The method of Tsitouras is tested quite well in practice. But maybe something is wrong with it for this specific problem?
To double-check the results, we also apply the classical fifth-order method of Dormand and
Prince \cite{dormand1980family} well-known from \texttt{ode45} in MATLAB
\cite{shampine1997matlab}.
The short-term results shown in Figure~\ref{fig:system3_original_DP5}
are the same as before. However, this time the numerical solutions blow up instead.
\begin{figure}[ht]
\centering
    \begin{subfigure}{0.49\textwidth}
        \includegraphics[width=\textwidth]{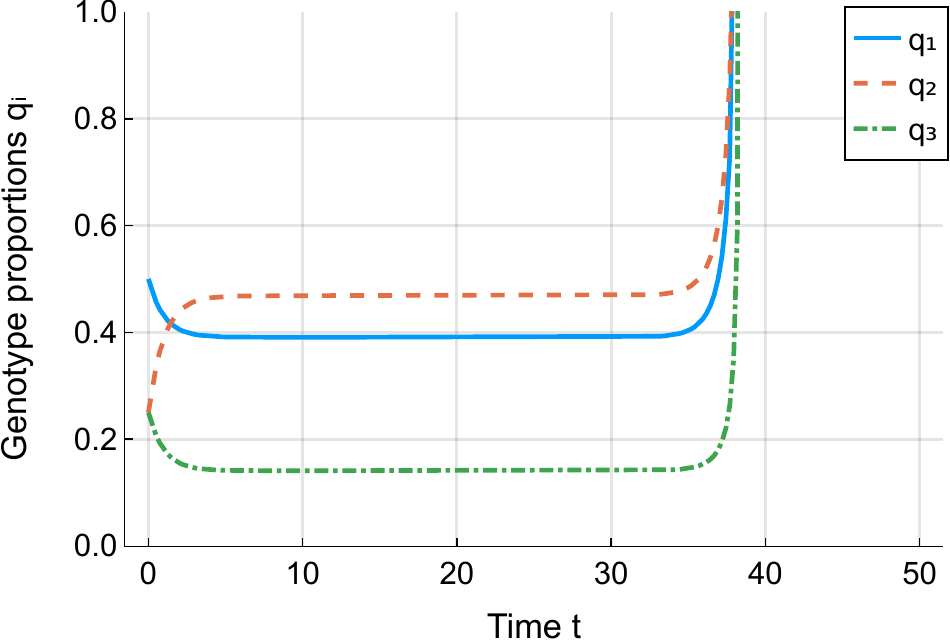}
        \caption{Zoom-in.}
    \end{subfigure}%
    \hspace*{\fill}
    \begin{subfigure}{0.49\textwidth}
        \includegraphics[width=\textwidth]{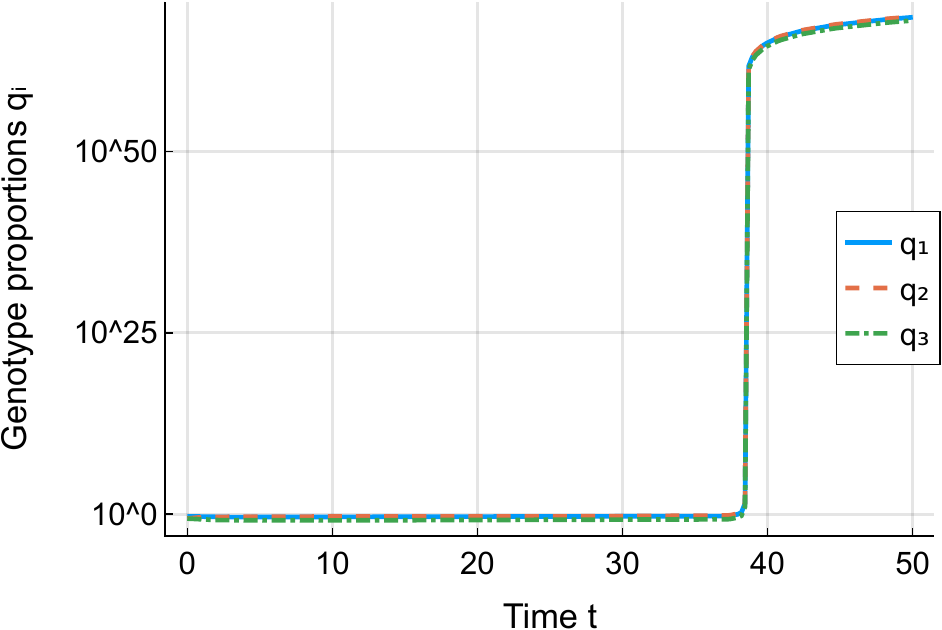}
        \caption{Zoom-out.}
    \end{subfigure}%
    \caption{Numerical solution of the system \eqref{eq:system3-original}
             with initial condition $q_0 = (0.5, 0.25, 0.25)^T$ obtained by the
             fifth-order Runge-Kutta method of Dormand and Prince \cite{dormand1980family}
             implemented in OrdinaryDiffEq.jl \cite{rackauckas2017differentialequations}
             in Julia \cite{bezanson2017julia} with absolute and relative tolerances
             $10^{-8}$.
             (a) Focus on the tendency towards steady states, (b) Change of the system behavior to blow up. }
    \label{fig:system3_original_DP5}
\end{figure}

The results for both methods do not change significantly if we use stricter tolerances.
We have also tested other methods (explicit and implicit Runge-Kutta methods, Rosenbrock-type
methods, multistep methods) implemented in a variety of software packages and programming
languages. The results are all similar to the two prototypical examples shown above --- both
contradicting the behavior we expected from the models. This is certainly something we need to
understand, in particular for studying the fixation or loss of alleles or genotypes over a longer timeframe.

\section{Dynamical system analysis}

As numerical simulations show surprising behavior of the solutions, we have a deeper look at the analysis of the dynamical systems.
The aim is to understand why the numerical simulations fail and to get ideas on how to overcome these difficulties by changing the system.

\subsection{A reduced 2-component model}

For some easier visualization, we regard an inheritance model for only
two genotypes ($q_1$ and $q_2$) and one allele, including some mutation.
In the example of studying flower colors, the two genotypes could be yellow and orange flowers that do not mix their color.
The mutation describes an outcome of genotype $q_i$ if the parent generation has the genotype $q_j$ with $i\neq j$. 
Given a parameter $a\in (0,1)$, the dynamics read
 \begin{equation}
\label{eq:system2-original}
\begin{aligned}
    q_1' &= f_1(q)= a q_1^2 + q_1 q_2 + (1-a) q_2^2 - q_1, \\
    q_2' &= f_2(q)= (1-a) q_1^2 + q_1 q_2 + a q_2^2 - q_2.
\end{aligned}
\end{equation}
The vector field visualization in Figure~\ref{fig:system2_original_vectorfield} illustrates the dynamics of the system.

The system \eqref{eq:system2-original} has two steady states, $(0,0)^T$ and $(1/2, 1/2)^T$,
where only the latter fulfills the conservation condition $q_\mathrm{sum}=q_1+q_2=1$ and therefore lies on the hyperplane  given by $q_2=1-q_1$.
This line represents an invariant manifold for $q_\mathrm{sum}$ fulfilling the conservation condition. 
The trivial steady state  $q^\star = 0$ is stable while the non-trivial steady state is unstable.
Figure~\ref{fig:system2_original_vectorfield} shows this behavior since initial values with $q_1 + q_2 < 1$ have trajectories leading to the trivial steady state.
Initial conditions with $q_1+q_2>1$ have trajectories with $\lim_{t \to \infty} q_i(t)= \infty$.

\begin{figure}[ht]
\centering
    \includegraphics[width=0.5\textwidth]{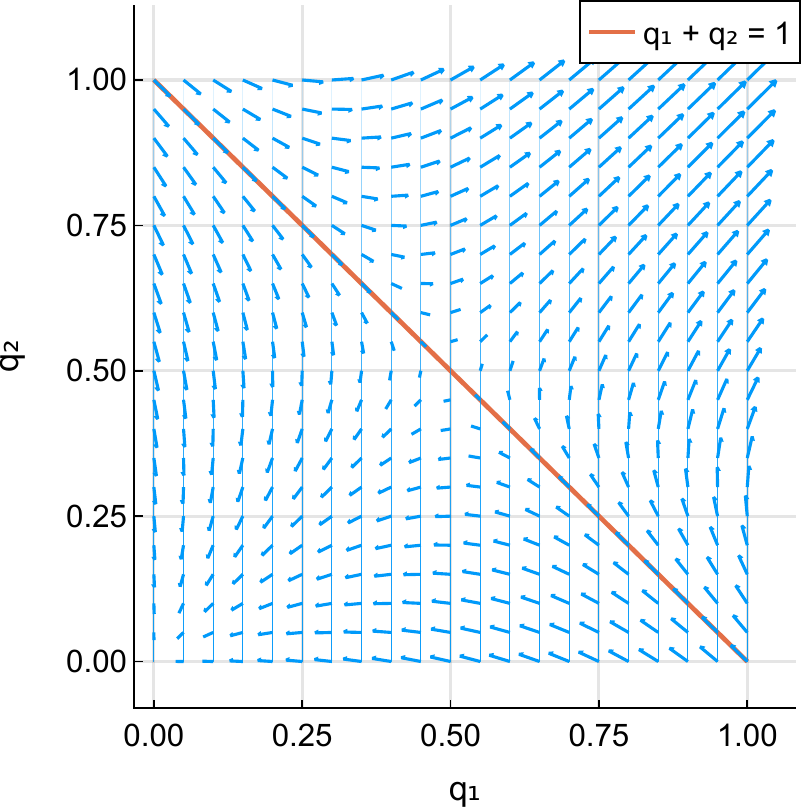}
    \caption{Visualization of the  
        vector field $f$
             of the 2-component system
             \eqref{eq:system2-original} with $a = 0.7$.
             The invariant manifold given by $q_1 + q_2 = 1$
             is highlighted.
             The length of the lines is proportional to
             $\sqrt{\|f(q)\|}$.}
    \label{fig:system2_original_vectorfield}
\end{figure}

Numerical solutions of \eqref{eq:system2-original} behave qualitatively
similar to their analogs for the 3-component model \eqref{eq:system3-original}:
The numerical solutions shown in Figure~\ref{fig:system2_original_Tsit5_Vern6}
seem to converge to the (unstable) steady state $(1/2, 1/2)^T$ at first but
eventually go to zero (the stable steady state) or blow up.

\begin{figure}[htb]
\centering
    \begin{subfigure}{0.49\textwidth}
        \includegraphics[width=\textwidth]{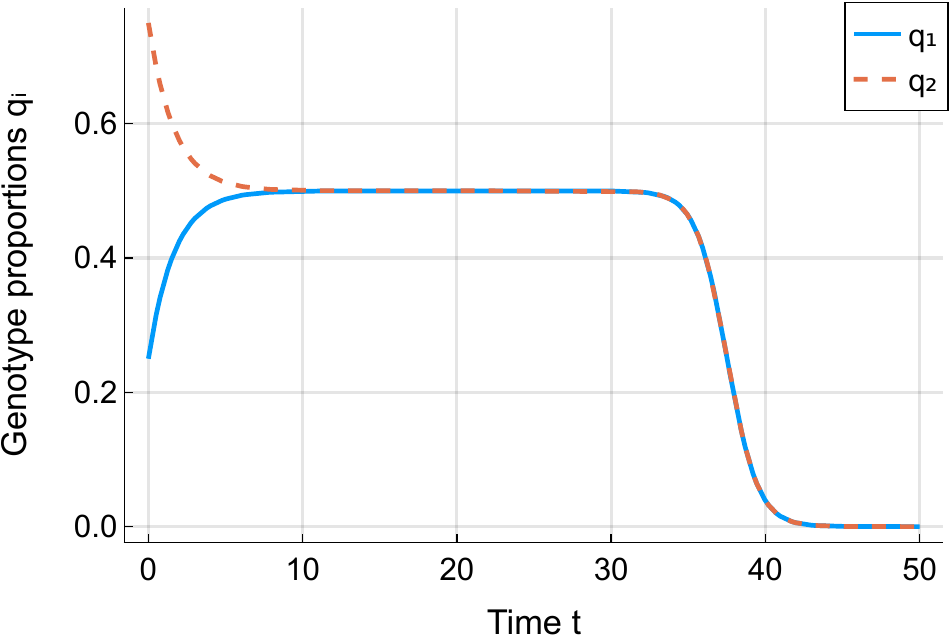}
        \caption{Tsitouras.}
    \end{subfigure}%
    \hspace*{\fill}
    \begin{subfigure}{0.49\textwidth}
        \includegraphics[width=\textwidth]{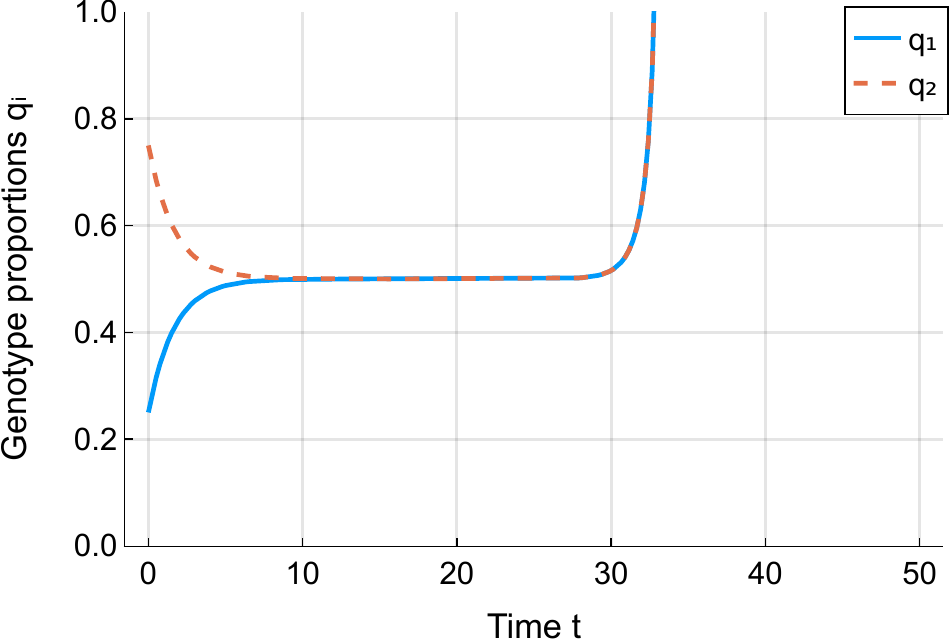}
        \caption{Verner.}
    \end{subfigure}%
    \caption{Numerical solution of the system \eqref{eq:system2-original}
             with $a = 0.7$ and initial condition $q_0 = (0.25, 0.75)^T$ obtained by the
             fifth-order Runge-Kutta method of Tsitouras \cite{tsitouras2011runge}
             and the sixth-order method of Verner \cite{verner2010numerically}
             implemented in OrdinaryDiffEq.jl \cite{rackauckas2017differentialequations}
             in Julia \cite{bezanson2017julia} with absolute and relative tolerances
             $10^{-8}$. For these tolerances, the method of Dormand and Prince
             \cite{dormand1980family} yields
             a numerical solution going to zero.}
    \label{fig:system2_original_Tsit5_Vern6}
\end{figure}

\subsection{Analysis of the 3-component model}

Now, we investigate the steady states of the system \eqref{eq:system3-original},
so $q^\star$ with $f_i(q^\star)=0$ for $i=1,2,3$.

Solving the equation gives a one-parameter family of
non-negative steady states
\begin{equation}
  q^\star=  \left( q_1^\star, 2\left( \sqrt{q_1^\star} - q_1^\star \right), 1 + q_1^\star - 2 \sqrt{q_1^\star} \right)^T,
    \qquad
    q_1^\star \in [0,1],
\end{equation}
fulfilling the conservation condition $\sum_{i=1}^3 q_i^\star = 1$ and reflecting the binomial distribution of the two alleles.
The trivial state $q^\star = 0 \in \mathbb{R}^3$ is a steady state as well but does not fulfill the conservation condition.

To investigate the stability properties of the steady states, we compute
the eigenvalues of the Jacobian
\begin{equation}
    f'(q)
    =
    \begin{pmatrix}
        -1 + 2 q_1 + q_2 & q_1 + q_2 / 2 & 0 \\
        q_2 + 2 q_3 & -1 + q_1 + q_2 + q_3 & 2 q_1 + q_2 \\
        0 & q_2 / 2 + q_3 & -1 + q_2 + 2 q_3
    \end{pmatrix}.
\end{equation}
Its eigenvalues are $\lambda_k = -1 + k \sum_{i=1}^3 q_i$, $k \in \{0, 1, 2\}$,
with associated eigenvectors
\begin{equation}
    v_0 = \begin{pmatrix} 1 \\ -2 \\ 1 \end{pmatrix},
    \quad
    v_1 = \begin{pmatrix} -(2 q_1 + q_2) \\ 2 (q_1 - q_3) \\ q_2 + 2 q_3 \end{pmatrix},
    \quad
    v_2 = \begin{pmatrix} (2 q_1 + q_2)^2 \\ 2 (2 q_1 + q_2) ( q_2 + 2 q_3) \\ (q_2 + 2 q_3)^2 \end{pmatrix},
\end{equation}
for $q \ne 0$; $f'(0) = -\operatorname{I}$ has the eigenvalue $-1$ with
multiplicity three.
In particular, the non-trivial steady states are unstable with
eigenvalues $-1$, $0$, and $1$. The trivial steady state
$q^\star = 0$ is asymptotically stable. All of these non-negative steady states
are visualized in
Figure~\ref{fig:system3_original_steadystates}.

\begin{figure}[htb]
\centering
    \includegraphics[width=0.75\textwidth]{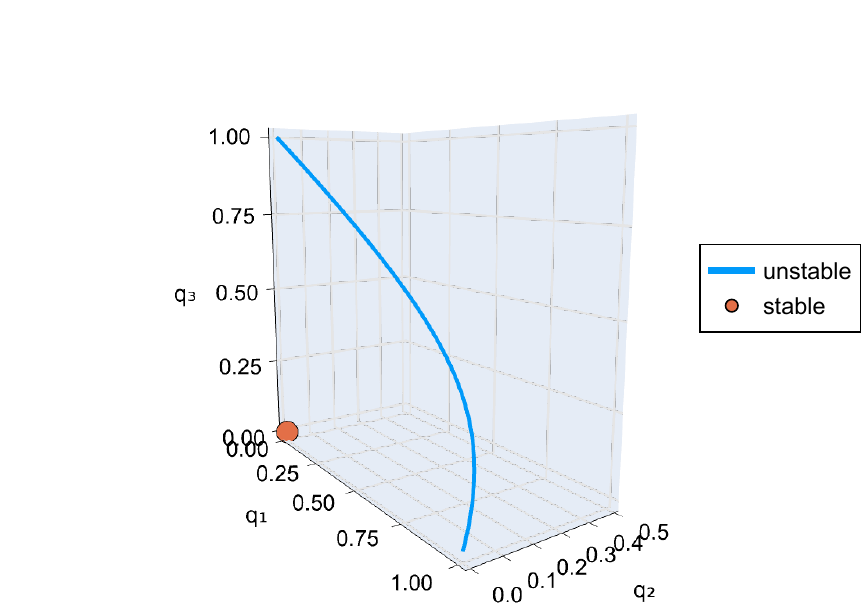}
    \caption{Non-negative steady states
             of the system \eqref{eq:system3-original}.
             All non-trivial steady states are unstable.}
    \label{fig:system3_original_steadystates}
\end{figure}

The dynamical system analysis allows us to answer the first question of why we are having unstable numerical simulations:
The steady states on the hyperplane with $\sum_i q_i =1$ are unstable, both for the 2-component system \eqref{eq:system2-original} and for the 3-component system \eqref{eq:system3-original}.
Useful tools for gaining these information are either the visualization of the dynamical system as a vector field or the analysis of the eigenvalues of the Jacobian.
The preferred method depends on the dimension of the problem.

Small perturbations occurring from floating point errors in the numerical simulations make the system leave the (unstable) hyperplane with $\sum_i q_i =1$ and the solutions tend to zero or blow up.

The next question is, how we can overcome these difficulties.

\section{Using invariants for gaining better models}

The systems \eqref{eq:system3-original} and \eqref{eq:system2-original} are constructed under the assumption that the total amount $\sum_i q_i =1$ is conserved.
From a modeling perspective, this property seems to remain unchanged for the differential equation system and therefore gives an invariant.
Numerical simulations show that $\sum_i q_i$ is not conserved during the simulations.
The analysis of the steady states even proves the instability of steady states fulfilling $\sum_i q_i =1$.

\subsection{Qualitative properties of invariants}

The question of how to overcome these problems leads us to the analysis of invariants.

\begin{definition}[First integrals, cf.\ Section~2.1 of \cite{goriely2001integrability}]
Let $U\subset \mathbb{R}^M$ be open. Consider the ODE
$q'= f(q)$ with $f\colon U \to \mathbb{R}^M$.
A (time-independent) \emph{first integral} of the ODE is a $C^1$-function
$J\colon U \to \mathbb{R}$ with $J'(q) \cdot f(q) = 0$ for all $q \in U$.
The first interval is called \emph{trivial} if
$\exists c \, \forall q \in U\colon J(q) = c$.
\end{definition}

Clearly, a first integral is an invariant of the ODE, i.e., it does not
change along a solution trajectory. Indeed, if $q$ solves the ODE
$q'(t) = f\bigl( q(t) \bigr)$,
\begin{equation}
    \frac{\dif}{\dif t} J\bigl( q(t) \bigr)
    =
    J'\bigl( q(t) \bigr) \cdot q'(t)
    =
    J'\bigl( q(t) \bigr) \cdot f\bigl( q(t) \bigr)
    =
    0.
\end{equation}

We first prove that the sum $\sum_i q_i$ is not a first integral for \eqref{eq:system3-original}.
The functional $J(q)=\sum_i q_i$ has the gradient $(1,1,1)^T$.  Thus,
\begin{equation}
\begin{aligned}
   J'(q) \cdot f(q)
   &= \left ( \frac{\partial }{\partial q} \sum_i q_i \right) \cdot f(q)
   =
    \sum_i f_i
    \\
    &=
    q_1^2 + q_2^2 + q_3^2 + 2 q_1 q_2 + 2 q_1 q_3 + 2 q_2 q_3 - q_1 - q_2 - q_3
    \\
    &=
    (q_1 + q_2 + q_3)^2 - (q_1 + q_2 + q_3)
    \neq 0
\end{aligned}
\end{equation}
in general. An analogous result holds for the 2-component model.
Thus, the systems \eqref{eq:system3-original} and \eqref{eq:system2-original}
do not conserve the sum $J(q)=\sum_i q_i$ for all chosen initial values.
But once we choose initial values fulfilling $\sum_i q_i =1$, the analytical
solution does not leave the level set of $J$.
Second integrals take this property into account:

\begin{definition}[Second integrals, cf.\ Section~2.5 of \cite{goriely2001integrability}]
    Let $U\subset \mathbb{R}^M$ be an open subset. Consider the ODE
    $q'= f(q)$ with $f\colon U \to \mathbb{R}^M$.
    A \emph{second integral} of the ODE $q'(t) = f\bigl( q(t) \bigr)$
    with vector field $f\colon U \to \R^M$ is a $C^1$-function
    $J\colon U \to \R$ such that there is a function
    $\alpha\colon U \to \R$ satisfying
    $\forall q \in U\colon J'(q) \cdot f(q) = \alpha(q) J(q)$.
\end{definition}

Clearly, second integrals $J$ are conserved by solutions of the ODE
whenever the initial conditions starts on the manifold given by
$J(q) = 0$.

The systems \eqref{eq:system3-original} and \eqref{eq:system2-original}
have the total sum $\sum_i q_i = 1$ as a second integral.
The difference $u := \sum_i q_i - 1$ between the total sum and unity satisfies
\begin{equation}
\label{eq:deviation-sum-qi}
    u'
    =
    \sum_i q_i'
    =
    \biggl( \sum_i q_i \biggr)^2 - \biggl( \sum_i q_i \biggr)
    =
    (1 + u)^2 - (1 + u)
    =
    (1 + u) u.
\end{equation}
All Runge-Kutta methods preserve such linear second integrals
\cite{tapley2021preservation}.

The ODE \eqref{eq:deviation-sum-qi} is very similar to the ODE
of logistic growth and can be solved analytically,
\begin{equation}
    u(t) = \frac{ u_0 }{ \e^{-t} (1 + u_0) - u_0 },
    \qquad
    u_0 = u(0).
\end{equation}
This ODE \eqref{eq:deviation-sum-qi} has two steady states,
the unstable steady state $u = 0$, corresponding to $\sum_i q_1 = 1$, and
the asymptotically stable steady state $u = -1$ corresponding to $\sum_i q_i = 0$.
For an initial condition $u_0 < 0$, $u(t) \to -1$ for $t \to \infty$.
For initial data $u_0 > 0$, $u(t)$ grows without bounds and blows up in finite time.
We discovered this behavior for the 2-component model in Fig.~\ref{fig:system2_original_Tsit5_Vern6}.

An initial condition of the original model
\eqref{eq:system3-original} needs to satisfy $\sum_i q_i = 1$.
Thus, the system \eqref{eq:deviation-sum-qi} is in the unstable
steady state. When solving \eqref{eq:system3-original} numerically,
it can be expected that even tiny floating point errors can accumulate
over time, either resulting in a blow-up in finite time or a convergence
to the biologically meaningless state $\sum_i q_i = 0$.
To verify this, we also performed computations using different floating
point types. The results shown in Figure~\ref{fig:system3_original_Float32_BigFloat}
are qualitatively the same as before, only the time changes when deviations from the
steady state become visible. 
\begin{figure}[htb]
\centering
    \begin{subfigure}{0.49\textwidth}
        \includegraphics[width=\textwidth]{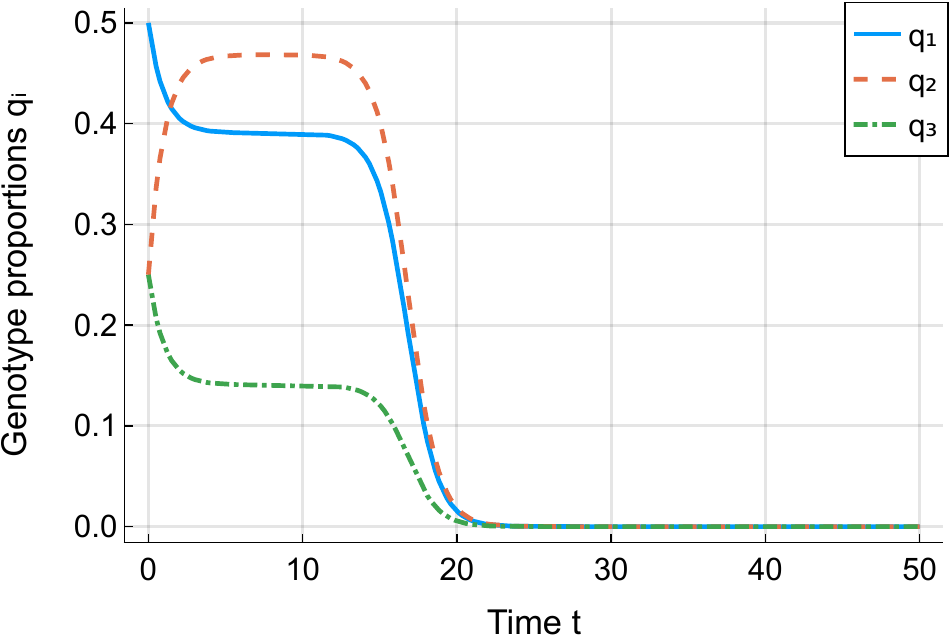}
        \caption{Tsitouras, \texttt{Float32}.}
    \end{subfigure}%
    \hspace*{\fill}
    \begin{subfigure}{0.49\textwidth}
        \includegraphics[width=\textwidth]{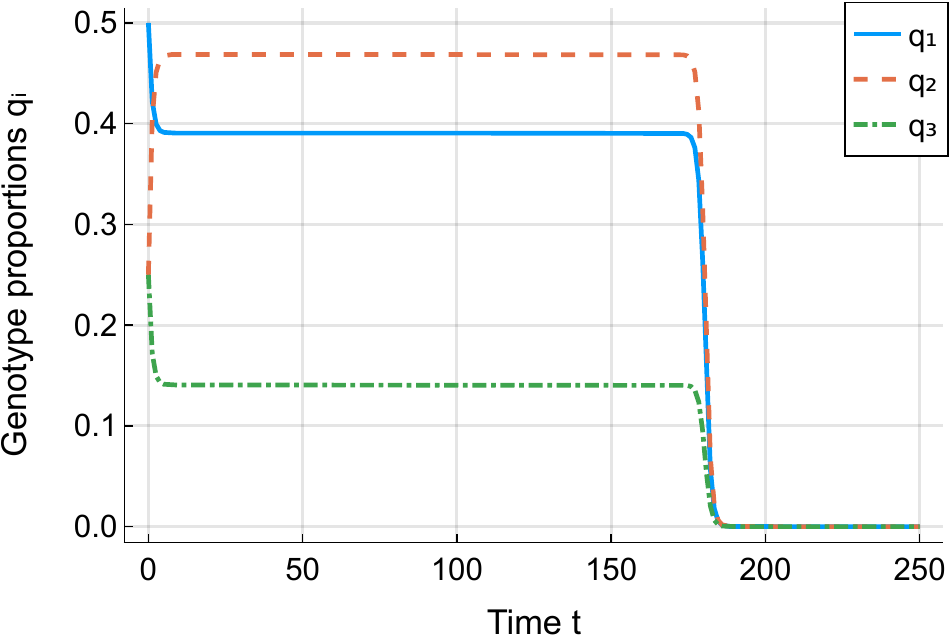}
        \caption{Tsitouras, \texttt{BigFloat}.}
    \end{subfigure}%
    \\
    \begin{subfigure}{0.49\textwidth}
        \includegraphics[width=\textwidth]{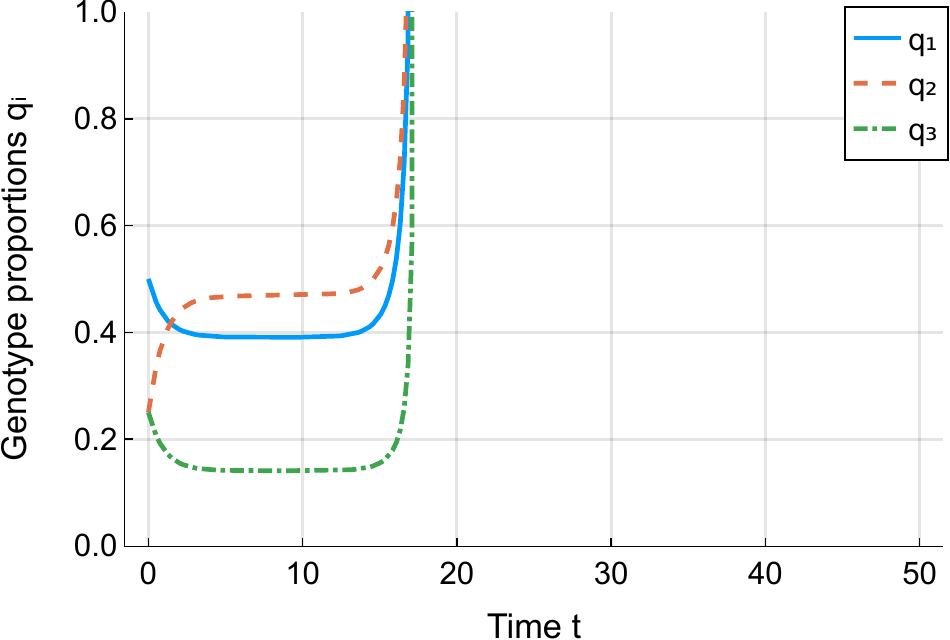}
        \caption{Dormand and Prince, \texttt{Float32}.}
    \end{subfigure}%
    \hspace*{\fill}
    \begin{subfigure}{0.49\textwidth}
        \includegraphics[width=\textwidth]{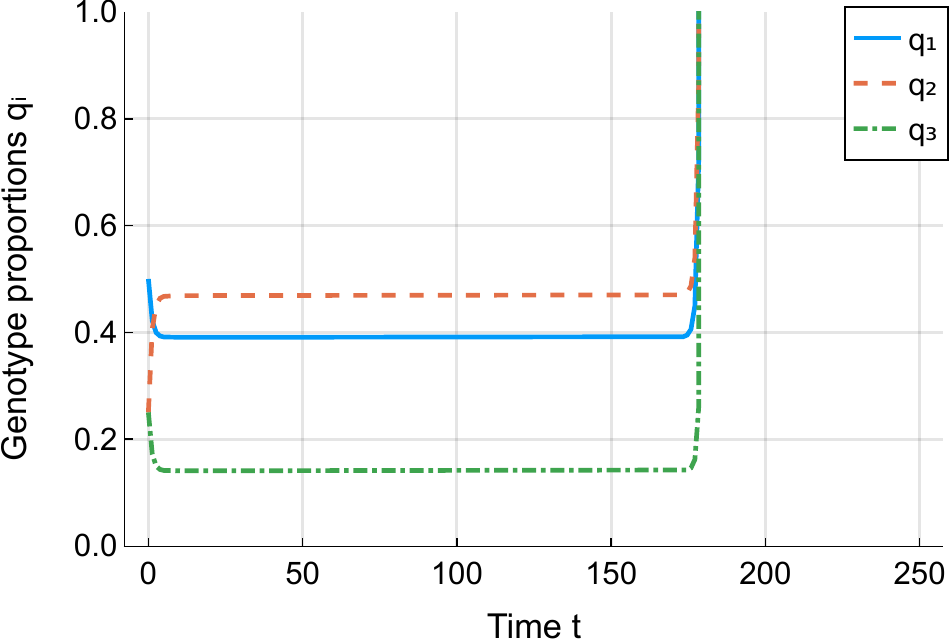}
        \caption{Dormand and Prince, \texttt{BigFloat}.}
    \end{subfigure}%
    \caption{Numerical solution of the system \eqref{eq:system3-original}
             with initial condition $q_0 = (0.5, 0.25, 0.25)^T$ obtained by the
             fifth-order Runge-Kutta methods of Tsitouras \cite{tsitouras2011runge}
             and Dormand and Prince \cite{dormand1980family}
             implemented in OrdinaryDiffEq.jl \cite{rackauckas2017differentialequations}
             in Julia \cite{bezanson2017julia} with different floating point types.
             For 32 bit floating point numbers \texttt{Float32} we use tolerances
             $10^{-7}$. For high-precision floating point numbers
             \texttt{BigFloat}, we use stricter tolerances $10^{-14}$.}
    \label{fig:system3_original_Float32_BigFloat}
\end{figure}

To overcome these issues, we reformulate the system so that the total sum
becomes a first integral. First, we present a general result showing that
this is possible.

\begin{theorem}
\label{thm:affine-second-to-first-integral}
    Consider an ODE $q'(t) = f\bigl( q(t) \bigr)$ with affine second integral $J$.
    Then, there is a modified vector field $\widetilde{f}$ having
    $J$ as first integral such that
    \begin{equation*}
        \forall y\colon \quad
        J(y) = 0 \implies f(y) = \widetilde{f}(y).
    \end{equation*}
    In particular, the original ODE is equivalent to the modified
    ODE $\widetilde{q}'(t) = \widetilde{f}(\widetilde{q}(t)))$
    whenever the initial condition is in the zero set of $J$.
   The modified vector field $\widetilde{f}$ can be constructed as $\widetilde{f}(y) = f(y) -\alpha(y) J'(y) / \| J'(y) \|^2 J(y)$, where we interpret the gradient $J'(y)$ as vector field.
\end{theorem}
\begin{proof}
    We have $J'(y) \cdot f(y) = \alpha(y) J(y)$. Choose the ansatz
    $\widetilde{f}(y) = f(y) + J(y) c$ where $c$ is a constant vector that will be suitably chosen. 
    Then,
    \begin{equation*}
        J'(y) \cdot \widetilde{f}(y)
        =
        J'(y) \cdot f(y) + J(y) J'(y) \cdot c
        =
        \left( \alpha(y) + J'(y) \cdot c \right) J(y).
    \end{equation*}
    Since $J$ is affine, $J'(y)$ is constant.
    If $J'(y) = 0$, $J(y)=c$ and $J$ is already a (trivial) first integral and $c$ can be chosen arbitrarily.
    Otherwise, choose, e.g., $c = - \alpha(y) J'(y) / \| J'(y) \|^2$
    and obtain $J'(y) \cdot \widetilde{f}(y) = 0$.
    Thus, $J$ is a first integral of the vector field
    $\widetilde{f}$.
\end{proof}

We will see, that various choices of the coefficient $c$ used
in the proof above are possible, leading to different modified models. 
Further, the modification changes the set of steady states of the ODE leading to different numerical behavior. 

\subsection{Reformulation of the models}

The theory of invariants allows reformulations of the system.
The aim is to modify the system in such a way that the stability properties of the steady states improve while retaining the dynamic behavior.
We cannot expect to have asymptotically stable steady states but having stable steady states is already an improvement. 
Then, numerical simulations will be stable and floating point errors do not change the qualitative behavior of the system.

\subsubsection{Reformulation of the 2-component model}

We start with reformulating the system \eqref{eq:system2-original} intuitively by using the property $q_1+q_2 =1$. First, we reorder the system as
\begin{equation}
    q_1' = a q_1^2 + q_1 \underbrace{(q_2 - 1)}_{= -q_1} + (1 - a) q_2^2,
    \quad
    q_2' = (1 - a) q_1^2 + \underbrace{(q_1 - 1)}_{= -q_2} q_2 + a q_2^2.
\end{equation}
Replacing $q_2-1=-q_1$ in the first equation and $q_1-1=-q_2$ in the second gives the new system
\begin{equation}
\label{eq:system2-modified}
\begin{aligned}
    q_1' &= (1-a) \bigl( - q_1^2 + q_2^2\bigr), \\
    q_2' &= (1-a) \bigl( q_1^2 -  q_2^2 \bigr).
\end{aligned}
\end{equation}
In this system, the sum of the components is a first integral since $q_1'+q_2'=0$ without any further assumptions.

\begin{remark}\label{rem:changes}
    The process of changing a dynamical system such that an (affine)
    second integral becomes a first integral is not unique. 

    A second integral for system \eqref{eq:system2-original} is given by $J= u=\sum_i q_i -1$ in \eqref{eq:deviation-sum-qi}. Indeed,
    choosing the function $\alpha(q)= \sum_i q_i$
    yields the defining equation 
    $J'(q) \cdot f(q) = \alpha (q) J(q)$. 
    Based on the second integral $J$ we reformulate the system as $\widetilde{f}(q) = f(q) + J(q) c$.

    Our intuitive approach leading to system \eqref{eq:system2-modified} uses the special choice $c=-q$. 
    We prove by a calculation that the modified system has a first integral: 
    The function $\alpha(q)= \sum_i q_i$ can be written as $\alpha (q) = J'(q) \cdot q$ with $J(q) = \sum_i q_i -1$. 
    The new vector field $\widetilde{f}$ is given as $\widetilde{f}(q)=f(q)- J(q)q$ and the new ODE has a first integral because 
        \begin{equation}
             \frac{\dif}{\dif t} J\bigl( q(t) \bigr) = J'(q) \cdot \widetilde{f}(q) = J'(q) \cdot f(q) - J'(q) \cdot J(q) q = \alpha (q) J(q) - J'(q) \cdot J(q) q = 0,
        \end{equation}
        where we first use the definition of $\widetilde{f}$, then the second integral property of $f$ and thirdly the observation that $\alpha (q) = J'(q) \cdot q$. 
    
    On the other hand,
    the modified vector field $\widetilde{f}$ constructed 
    in the proof of Theorem~\ref{thm:affine-second-to-first-integral}
    with $c=- \alpha(q) J'(q) / \| J'(q) \|^2$
    yields the ODE
    \begin{equation}
    \begin{aligned}
        q_1'
        &=
        \left( a q_1^2 + q_1 q_2 + (1 - a) q_2^2 - q_1 \right) -
            \frac{1}{2} \left( (q_1 + q_2)^2 - (q_1 + q_2) \right)
        \\
        &=
        (a - 1/2) (q_1^2 - q_2^2) - \frac{1}{2} (q_1 - q_2),
        \\
        q_2'
        &=
        \left( (1 - a) q_1^2 + q_1 q_2 + a q_2^2 - q_2 \right) -
        \frac{1}{2} \left( (q_1 + q_2)^2 - (q_1 + q_2) \right)
        \\
        &=
        (1/2 - a) (q_1^2 - q_2^2) + \frac{1}{2} (q_1 - q_2),
    \end{aligned}
    \end{equation}
    for the 2-component model \eqref{eq:system2-original}. This is
    clearly different from \eqref{eq:system2-modified} in general
    (but of course equivalent on the invariant manifold).
\end{remark}

We continue our investigations with the new system \eqref{eq:system2-modified}.
This system has the non-negative steady states $(q_1^\star, q_1^\star)^T$ with $q_1^\star \geq 0$.
For $q_1^\star \ne 1/2$, the steady states do not fulfill the condition $\sum_{i=1}^2 q_i = 1$.
The vector field is shown in Figure~\ref{fig:system2_modified_vectorfield}.
For any initial conditions, the dynamics lead to a steady state $(q_1^\star, q_1^\star)^T$ with $q_1^\star= 1/2 \, ( q_1(0) + q_2(0))$.

\begin{figure}[ht]
\centering
    \includegraphics[width=0.5\textwidth]{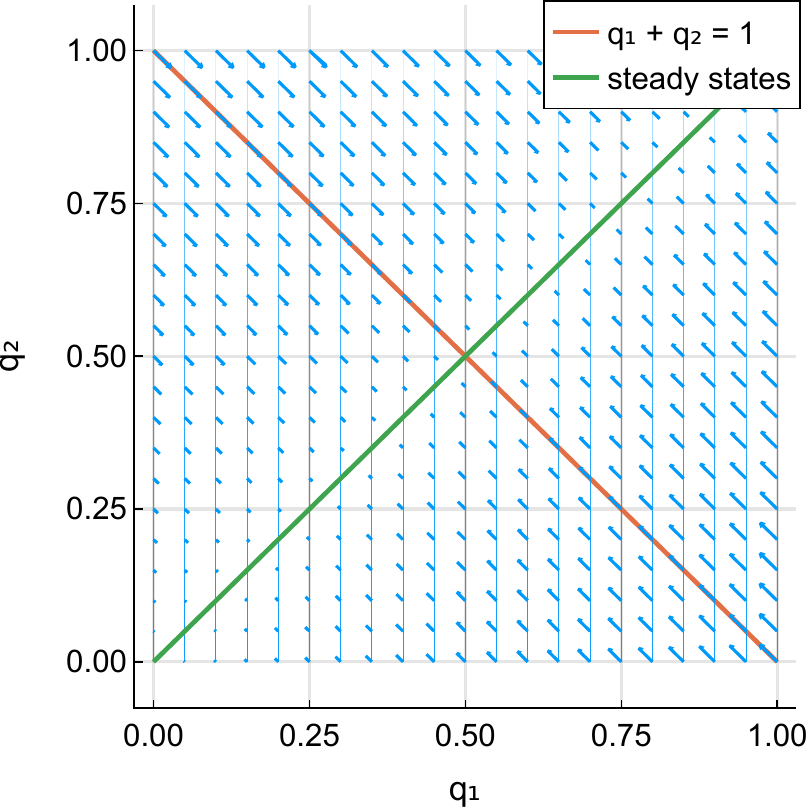}
    \caption{Visualization of the right-hand side vector field $f$
             of the modified 2-component system
             \eqref{eq:system2-modified} with $a = 0.7$.
             The invariant manifold given by $q_1 + q_2 = 1$
             is highlighted.
             The length of the lines is proportional to
             $\sqrt{\|f(q)\|}$.}
    \label{fig:system2_modified_vectorfield}
\end{figure}

The stability of the steady states can be assured by calculating the eigenvalues of the Jacobian of the system.
The Jacobian
\begin{equation}
   f'(q)
    =
    2 (1 - a)
    \begin{pmatrix}
        -q_1 & q_2 \\
        q_1 & -q_2
    \end{pmatrix}
\end{equation}
has the eigenvalues $\lambda_0 = 0$ and $\lambda_1 = -2 (1 - a) (q_1 +q_2) < 0$
with eigenvectors
\begin{equation}
    v_0 = \begin{pmatrix} q_2 \\ q_1 \end{pmatrix},
    \quad
    v_1 = \begin{pmatrix} -1 \\ 1 \end{pmatrix},
\end{equation}
for $q \ne 0$; $f'(0) = 0$ has the double eigenvalue $0$.
Consequently, the non-trivial steady states are stable.

Starting with any initial values $q_1, q_2>0$, the solutions tend towards the steady states.
If the initial conditions fulfill the condition $\sum_i q_i=1$ of the hyperplane, the solutions remain on the hyperplane.
Deviations on the manifold, for example caused by floating point errors, will not change the overall system behavior and result in deviations of the steady state from the hyperplane.
For this model, a transformation of variables to $p_1=q_1+q_2$ and $p_2=q_2-q_1$ separates the tendency towards the hyperplane from the stationary behavior on the manifold but would change the interpretability of the variables.

\begin{remark}
    The stability of steady states of the modified 2-component
    system \eqref{eq:system2-modified} can be understood easily
    by looking at Figure~\ref{fig:system2_modified_vectorfield}.

    This behavior is also generic  in the sense that it can be
    obtained from the center manifold theorem --- at least locally,
    see \cite[Theorem~I.4]{iooss1998topics} or
    \cite[Theorem~2.7]{marsden1976hopf}.
    Summarized briefly, the center manifold theorem states that a
    dynamical system near a steady state where all eigenvalues of
    the Jacobian have a non-negative real part behaves as follows:
    i) There is a neighborhood $V$ of the steady state and a
    manifold $M$ of the same dimension as the generalized eigenspace associated
    with eigenvalues on the imaginary axis that is invariant under
    the flow as long as the evolving state stays in $V$.
    ii) If the state stays in $V$ for all times, it converges
    (exponentially) to the manifold $M$.
    In our case, the eigenvalue on the imaginary axis is zero and
    the corresponding manifold is the set of steady states.
\end{remark}

In this particular example, a stable model can be derived as well by including the manifold equation $1=q_1+q_2$ directly into \eqref{eq:system2-modified}. Then, the system becomes decoupled and $q_1'=(1-a)(1-2q_1)$ is a single equation describing the evolution without any interpretable mechanisms. This reduction step does not follow the intention of highlighting the interplay between modeling, analysis and numerical simulation, so we will not discuss it further.

\subsubsection{Reformulation of the 3-component model}

Inserting the conservation condition 
$\sum_{i=1}^3 q_i = 1$ into the first equation of \eqref{eq:system3-original} yields
\begin{equation}
    q_1'
    = q_1^2 + q_1 q_2 + \frac{1}{4} q_2^2 - q_1 \underbrace{(q_1 + q_2 + q_3)}_{= 1}
    =
    \frac{1}{4} q_2^2 - q_1 q_3.
\end{equation}
Continuing this process for the other equations as well results in
the modified system
\begin{equation}
\label{eq:system3-modified}
\begin{aligned}
    q_1' &= \frac{1}{4} q_2^2 - q_1 q_3, \\
    q_2' &= -\frac{1}{2} q_2^2 + 2 q_1 q_3, \\
    q_3' &= \frac{1}{4} q_2^2 - q_1 q_3.
\end{aligned}
\end{equation}
Again, as in the case of the 2-component model, this reformulation can be interpreted as turning the second integral $J(q)=u(q)=\sum_i q_i -1$ with $\alpha(q)= \sum_i q_i$ into a first integral for the new vector field $\widetilde{f} = f-J(q) q$. 
The calculations in Remark~\ref{rem:changes} are valid as well for the 3-component system. 

The modified system \eqref{eq:system3-modified} has the non-negative
steady state
$q^\star = ( q_1^\star, q_2^\star, q_3^\star )^T$ with
\begin{equation}
    q_3^\star
    =
    \begin{cases}
        \text{arbitrary} \ge 0, & \text{if } q_1^\star = q_2^\star = 0, \\
        \frac{q_2^{\star 2}}{4q_1^\star}, & \text{if } q_1^\star \ne 0,
    \end{cases}
\end{equation}
for $q_1^\star, q_2^\star \ge 0$ and $q_2^\star = 0$ if $q_1^\star = 0$.
Note in particular that the steady states of \eqref{eq:system3-modified}
satisfying $\sum_i q_i = 1$ are exactly the non-trivial steady states of
the original system \eqref{eq:system3-original}.
The Jacobian is
\begin{equation}
    f'(q)
    =
    \begin{pmatrix}
        -q_3 & q_2 / 2 & -q_1 \\
        2 q_3 & -q_2 & 2 q_1 \\
        -q_3 & q_2 / 2 & -q_1 \\
    \end{pmatrix}.
\end{equation}
For $q \ne 0$, $f'(q)$ has a double eigenvalue $\lambda_0 = \lambda_1 = 0$
with eigenspace spanned by
\begin{equation}
    v_0 = \begin{pmatrix} q_1 \\ 0 \\ -q_3 \end{pmatrix},
    \quad
    v_1 = \begin{pmatrix} q_2 / 2 \\ q_1 + q_3 \\ q_2 / 2 \end{pmatrix},
\end{equation}
as well as the eigenpair
\begin{equation}
    \lambda_2 = -\sum_{i=1}^3 q_i,
    \quad
    v_2 = \begin{pmatrix} 1 \\ -2 \\ 1 \end{pmatrix}.
\end{equation}
Moreover, $f'(0) = 0$.
Thus, all non-negative steady states are associated with eigenvalues of
$f'(q)$ with non-positive real parts.
The non-negative steady states are visualized in
Figure~\ref{fig:system3_modified_steadystates}.

\begin{figure}[htb]
\centering
    \includegraphics[width=0.75\textwidth]{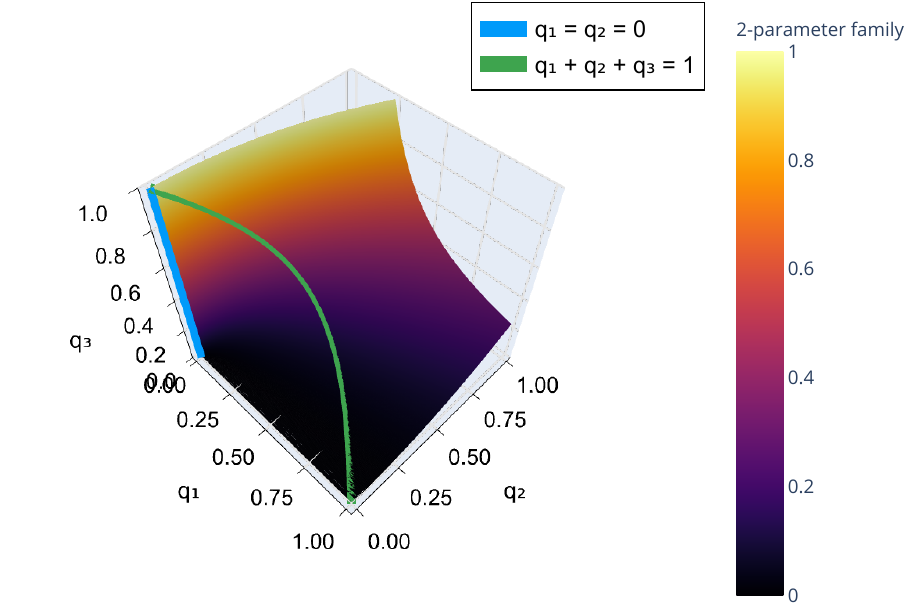}
    \caption{Non-negative steady states
             of the system \eqref{eq:system3-modified}.
             The system Jacobian has only eigenvalues with non-positive
             real parts at all steady states.}
    \label{fig:system3_modified_steadystates}
\end{figure}

As expected based on this stability analysis, numerical solutions behave properly when
applied to the modified systems \eqref{eq:system2-modified} and \eqref{eq:system3-modified}.
This is demonstrated in Figure~\ref{fig:systems_modified_Tsit5}.
\begin{figure}[htb]
\centering
    \begin{subfigure}{0.49\textwidth}
        \includegraphics[width=\textwidth]{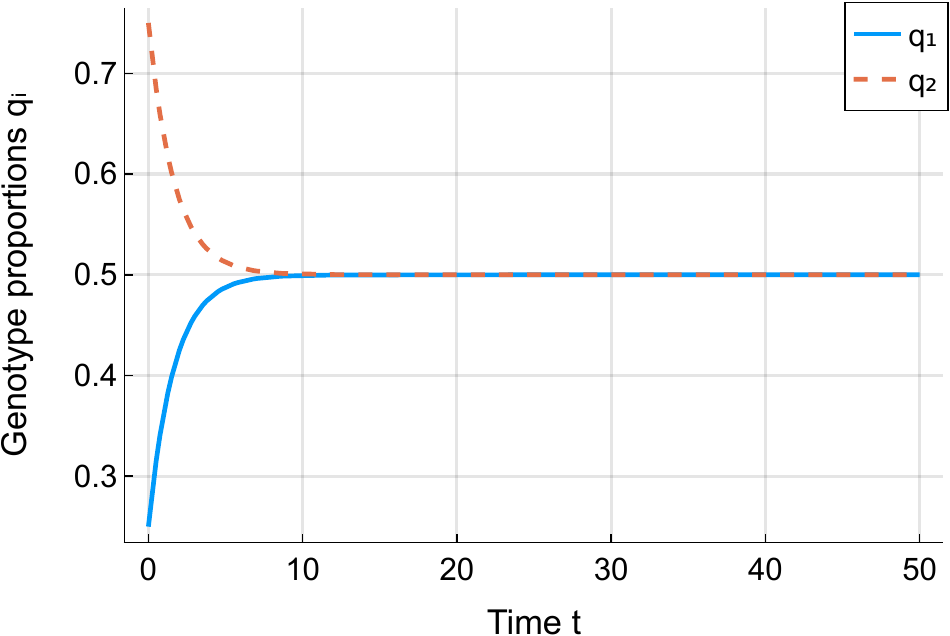}
        \caption{2-component model.}
    \end{subfigure}%
    \hspace*{\fill}
    \begin{subfigure}{0.49\textwidth}
        \includegraphics[width=\textwidth]{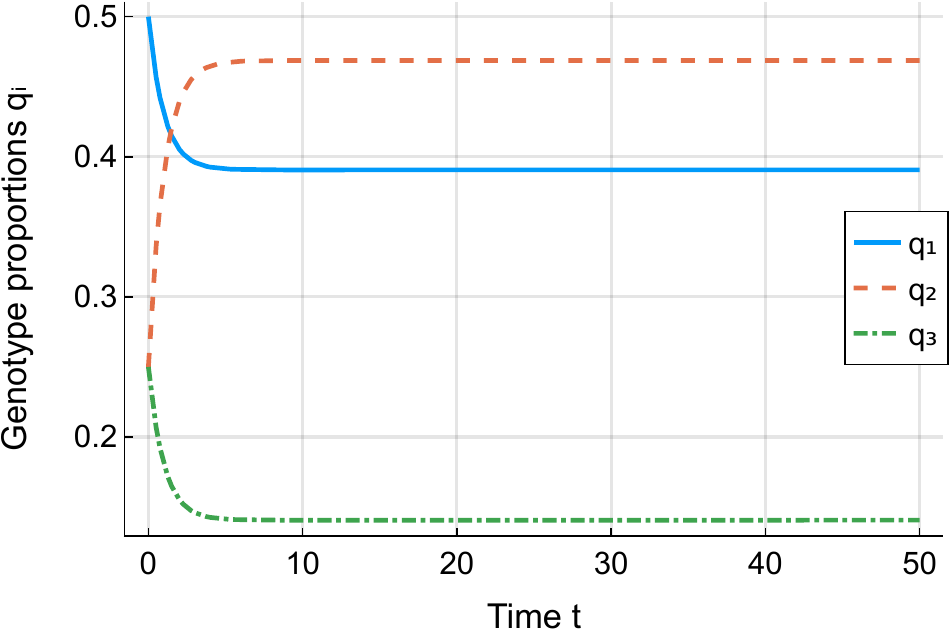}
        \caption{3-component model.}
    \end{subfigure}%
    \caption{Numerical solution of the modified systems \eqref{eq:system2-modified}
             (with $a = 0.7$ and initial condition $q_0 = (0.25, 0.75)^T$) and
             \eqref{eq:system3-modified} (with initial condition $q_0 = (0.75, 0.25, 0.25)^T$)
             obtained by the fifth-order Runge-Kutta method of Tsitouras \cite{tsitouras2011runge}
             implemented in OrdinaryDiffEq.jl \cite{rackauckas2017differentialequations}
             in Julia \cite{bezanson2017julia} with absolute and relative tolerances
             $10^{-8}$.}
    \label{fig:systems_modified_Tsit5}
\end{figure}

The modified systems \eqref{eq:system2-modified} and \eqref{eq:system3-modified} have stable steady states fulfilling $\sum_i q_i^\star=1$.
The modification of the reaction functions changes the system's properties from having a second integral to having a first integral.

\section{Evaluation of the models}

After improving the analytical properties and ensuring the conservation of
$\sum_i q_i$ during the numerical simulations, we like to draw attention to the modeling process.

The models \eqref{eq:system3-original} and \eqref{eq:system2-original} were derived quite directly from the biological application.
The system quantities $q_k$ give the proportion of genotype $k$ in the total population.
The condition $\sum_i q_i =1$ arises therefore naturally from the biological application and the mathematical translation into equations.
The biologically motivated models surprised us during the numerical simulations with their unstable behavior arising from leaving the hyperplane with $\sum_i q_i =1$.
Floating point errors already lead to unstable behavior, resulting in solutions tending towards zero or blowing up.

We used the conservation property to reformulate the ordinary differential equation systems.
The new systems \eqref{eq:system2-modified} and \eqref{eq:system3-modified} have the same dynamics on the hyperplane as the original models \eqref{eq:system3-original} and \eqref{eq:system2-original}.

Summarized, the new models were derived based on analytical considerations and not according to a modeling process. In general, changing the functions of a model might result in differential equations
without any (obvious and direct) connection to the modeling object.
Here, we try to interpret the inheritance of genes in the new model
\eqref{eq:system3-modified}.

\begin{figure}[htb]
\centering
    \includegraphics[width=0.75\textwidth]{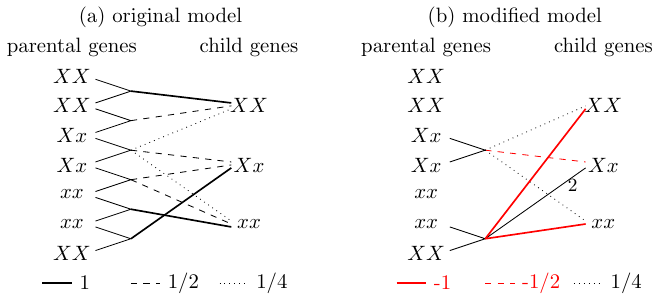}
    \caption{Modeled combinations of genes in the original system \eqref{eq:system3-original} and the modified system \eqref{eq:system3-modified}. }
    \label{fig:inheritance}
\end{figure}

While the system \eqref{eq:system3-original} translates all combinations of parental genes into equations, the modified system \eqref{eq:system3-modified} regards fewer combinations.
The modified system includes only the combination of the mixed genotype $Xx$ and the combination of the two pure genotypes.
The combination of two pure genotypes of the same type is not modeled.
This can be interpreted as neglecting combinations that do not change the proportions of genotypes in the population.
This affects as well the combination $q_2$ of two genotypes, since the factor $-1/2$ is the difference between the actual factor $1/2$ in system \eqref{eq:system3-original} and the factor 1 symbolizing the preservation of genotype $q_2$.

The change from system \eqref{eq:system3-original} to the modified system \eqref{eq:system3-modified} is a change in the point of view:
While the first model describes the random assortment of genotypes from the point of view of the parental generation, the new model focuses on the change of proportions.
In our case, only those combinations are included in the model that change the proportions of the genotypes in the population.
Leveling effects, like the inheritance of genotype $q_1$ with probability $1/2$ if $q_1$ and $q_2$ are combined and therefore $1/2$ of the parental genes are identical to the child genes, are not modeled.

A similar interpretation is possible for the 2-component model.
As the modified system \eqref{eq:system2-modified} only regards the change of the proportions, the pure conservation of one proportion is not regarded.

Interpretations like this are not automatically possible if a model with a second integral is turned into a model with a first integral.
In any case, the new formulation might reveal new perspectives on the application and the most relevant mechanisms in the mathematical description.

\section{Conclusion --- or what can we learn?\texorpdfstring{\nopunct}{}}

The example we discussed in this article shows how surprising numerical simulation results
may lead to insightful analysis and reformulation of mathematical models.
The reformulation gives insight into properties of the modeled application on two levels. 
First, the reformulation process itself sheds light on conservation properties of the model. 
Second, the new formulation may change the view on the model leading to new interpretations of the mechanisms. 
Gaining insight into underlying principles is complicated for large models. 
The proposed reformulation workflow may support the finding of hidden principles while improving the analytical and numerical properties.

Please note that we neither blame the numerics nor the modeling; it is
their combination that lead to the catastrophic behavior we discussed.
If it was affordable to use exact representations of all numbers occurring
in the numerical time integration process, we could just use the original
model. However, this is only possible for small problems and short time
scales, e.g., by using rational numbers with high-precision integer types
such as \texttt{Rational\{BigInt\}} in Julia.
However, it appears to be preferable for us to adapt the modeling process,
thereby avoiding the unrealistic instabilities inherent in the original model. This is
required for most practical purposes, e.g., if uncertainties are present
or larger-scale simulations are used.

To recap, a surprising behavior of numerical methods applied to an innocently
looking model from mathematical biology for studying genetic drift guided us on a tour through several
undergraduate and graduate courses including mathematical modeling, dynamical
systems, and numerical analysis. In particular, we had a chance to see the
importance of invariants and their manifestation in first and second integrals
as well as stability properties of steady states.

We would like to see this article as a bridge gathering results and techniques
from several basic courses together while motivating advanced
courses in applied mathematics. This includes for example deeper concepts
of stability analysis, the study of asymptotic behavior of dynamical systems \cite{perko_differential_2001, seydel_practical_2010},
and advanced results such as the center manifold theorem
\cite{iooss1998topics,marsden1976hopf} or Fenichel's theory on invariant manifolds \cite{fenichel_persistence_1971} on the theoretical side.
The latter emphasizes using transformation to local coordinates on the manifold. 
While this approach has many analytical benefits, local coordinates are oftentimes difficult to interpret in the context of applications. 

From the point of view of a numerical analyst, it stresses the importance
of structure-preserving numerical methods, introduced in textbooks
such as \cite{sanzserna1994numerical,hairer2006geometric} in the context
of numerical methods for ODEs; the corresponding partial differential equation topics appear to be
more specialized and available mostly in form of journal articles such as
\cite{tadmor2003entropy,egger2019structure,ranocha2021broad,ranocha2020relaxation}.
The basic reason for us to modify the mathematical models was to
stabilize the dynamics when the total sum is not equal to unity.
This is related to invariant-preserving numerical methods such as
the orthogonal projection approach described in
\cite[Section~IV.4]{hairer2006geometric}. However, most numerical
analysts would probably not consider using such tools here
since we only need to preserve a linear functional --- which the
standard time integration methods we employ do automatically
\cite{tapley2021preservation} --- at least in exact arithmetic.
Thus, the behavior is somewhat surprising but of course well 
understood after performing a stability analysis. In some sense,
we have to deal with non-ideal versions in practice while the common
analysis assumes ideal numerical methods. This happens also in many
other cases in practice, e.g., if iterative numerical methods are
used to solve equations arising in implicit time integration schemes
\cite{birken2022conservation,linders2022locally,linders2023resolving}.

We introduced the problem of invariants and manifolds in the context of mathematical biology. 
Of course, problems of this type occur in many applications, for example in control problems \cite{wehage_generalized_1982} or mechanical systems \cite{hairer2006geometric}. 
Finally, we encourage researchers in the field of mathematical modeling and applied sciences to discover the strength of mathematical theory, especially for dynamical systems and for structure-preserving numerical schemes, for gaining insight into the dynamics of interest. 

To make this study reproducible and useful for classroom teaching, we provide
all source code and instructions available online in our repository
\cite{reisch2023modelingRepro} --- in the form of an interactive notebook
using the modern programming language Julia.

\section*{Acknowledgments}

We would like to thank Dirk Langemann
for discussions about this topic and constructive 
comments on an early draft of the manuscript.

HR was supported by the Deutsche Forschungsgemeinschaft
(DFG, German Research Foundation, project number 513301895)
and the Daimler und Benz Stiftung (Daimler and Benz foundation,
project number 32-10/22).

\printbibliography

\end{document}